\documentclass[12pt]{article}

\usepackage{verbatim}
\usepackage{mathpazo}
\usepackage{graphicx}
\usepackage{amsfonts}
\usepackage{amssymb}
\usepackage{amsmath}
\usepackage{url}

\usepackage{lmodern}
\usepackage[a4paper]{geometry}
\usepackage{amsthm}

\theoremstyle{plain}
\newtheorem{thm}{Theorem}
\newtheorem{prop}{Proposition}

\newtheorem{lem}{Lemma}
\newtheorem{defi}{Definition}
\newtheorem{fact}{Fact}

\newcommand{\R}{\ensuremath{\mathbb{R}}} 

\theoremstyle{remark}

\newtheorem{rmq}{Remark}

\usepackage{amsmath,amssymb,amsfonts}

\title{Dimensional transport inequalities and Brascamp-Lieb inequalities}
\author{Erik Thomas}
\begin{document}
\maketitle

\begin{abstract}
The goal of the present paper is to discuss new transport inequalities for convex measures. We retrieve some dimensional forms of Brascamp-Lieb inequalities. We also give some quantitative forms involving the Wasserstein's distances.
\end{abstract}

\section{Introduction}

We shall begin by recalling Borell's terminology~\cite{bor1, bor2} about convex measures. Although we will not use explicitly Borell's results, it allows to explain the values and internal relations between the parameters appearing in our study. 

Let $\alpha \in \left[ - \infty , +\infty \right].$ A Radon probability measure $\mu$ on $\R^n$ (or on an open convex set $\Omega \subseteq \R^n$) is called $\alpha$-concave, if it satisfies
\begin{equation}
\mu \left( t A + \left( 1 - t \right) B \right) \ge \left(  t \mu \left( A \right)^{\alpha} + \left( 1 - t \right) \mu \left( B \right)^{\alpha} \right)^{\frac1\alpha},
\label{concave}
\end{equation}
for all $t \in \left( 0 , 1 \right)$ and for all Borel sets $A, B \subset \R^n.$ When $\alpha=0,$ the right-hand side of~\eqref{concave} is understood as $\mu \left( A \right)^t \mu \left( B \right)^{1 - t}$: $\mu$ is a log-concave measure. When $\alpha=-\infty,$ the right-hand side is understood as $\min \left\{ \mu \left( A \right) , \mu \left( B \right) \right\}$ and when $\alpha=+\infty$ as $\max \left\{  \mu \left( A \right) , \mu \left( B \right)  \right\}.$ We remark that the inequality~\eqref{concave} is getting stronger when $\alpha$ increases, so the case $\alpha=-\infty$ describes the largest class whose members are called convex or hyperbolic probability measures.  In~\cite{bor1, bor2}, Borell proved that a measure $\mu$ on $\R^n$ absolutely continuous with respect to the Lebesgue measure is $\alpha$-concave (and verifies~\eqref{concave}) if and only if $\alpha \le \frac1n$ and $\mu$ is supported on some open convex subset $\Omega \subseteq \R^n$ where it has a nonnegative density $p$ which satisfies, for all $t \in \left( 0 , 1 \right),$

\begin{equation}
p \left( t x + \left( 1 - t \right) y \right) \ge \left(  t p \left( x \right)^{\alpha_n} + \left( 1 - t \right) p \left( y \right)^{\alpha_n} \right)^{\frac{1}{\alpha_n}}, \qquad \forall x,y \in \Omega,
\label{concave2}
\end{equation}
where 
$\alpha_n:=\frac{\alpha}{1 - n \alpha} \in \left[ - \frac1n , + \infty \right]$. Note that this amounts to the concavity of $\alpha_n p^{\alpha_n}$.  In particular, $\mu$ is log-concave if and only if it has a log-concave density ($\alpha=\alpha_n=0$). 

We shall focus on the densities rather than on the measures, so let us reverse the perspective. We are given $\kappa >-1$, or more precisely,
\begin{equation}\label{defkappa}
\kappa \in \left[ - \frac1n , + \infty \right] = \left[-\frac1n, 0\right[\cup \{0\}\cup \left]0, +\infty\right[,
\end{equation}
and a probability density $\rho$ on $\R^n$ (by this we mean a nonnegative Borel function with $\int \rho = 1$), with the property that $\kappa \rho^{\kappa}$ is concave on its support. Borell's result then tells us that the (probability) measure with density $\rho
$ is $\kappa(n):=\frac{\kappa}{1+n\kappa}$-concave measure on $\R^n$. This  suggests two different behaviors depending on the sign of $\kappa$ since $\rho^{\kappa \left( n \right)}$ is convex or concave. Let us describe them. \\
{\bf Case 1} \\
This corresponds to  $\kappa \in ]0, +\infty]$ (that is, for measures, $0 < \kappa(n) \le \frac1n$).  We set $\beta:=\frac{1}{\kappa} \in \left[ 0 , +\infty \right)$ and we work with densities of the form $\rho_{\beta} \left( x \right)= \frac{W \left( x \right)^{\beta}}{\int_{\Omega} W^{\beta}}$ where $W : \R^n \to \R_+$ is  concave on its support. Note that the measure is supported on  $\Omega =\{W>0\}\subset \R^n$, which is  an open bounded convex set. The typical examples are the measures defined by

$$ d \tau_{\sigma, \beta}  \left( x \right)=\frac{1}{C_{\sigma, \beta}} \left( \sigma^2 - \left|x \right|^2 \right)_+^{\beta}  d x, \qquad \beta > 0, \sigma > 0,$$
where $C_{\sigma, \beta}=\int_{\R^n}  \left( \sigma^2 - \left|x \right|^2 \right)_+^{\beta}   d x =\sigma^{2\beta + n} \pi^{\frac{n}{2}} \frac{ \Gamma \left( \beta + 1 \right)}{\Gamma \left(  \beta + \frac{n}{2} + 1 \right)}$ is a normalizing constant.
{\bf Case 2} \\
This corresponds to $\kappa \in [-\frac1n, 0[$ (that is, for measures, $\kappa(n) \le 0$). We set $\beta:=- \frac{1}{\kappa}=n - \frac1{\kappa(n)} \ge n$ and we work with densities of the form $\rho_{\beta} \left( x \right)=\frac{ W \left( x \right)^{- \beta}}{\int_{\Omega}  W^{-\beta} }$ where $W: \R^n \to \R_+\cup\{+\infty\}$ is a convex function. Note that the support of the measure is given by the convex set $\{W<+\infty\}$. The typical examples are the (generalized) Cauchy probability measures defined by

$$ d \mu_{\beta} \left( x \right)=\frac{1}{C_{\beta}} \left( 1 + \left| x \right|^2 \right)^{- \beta}  d x, \qquad \beta > \frac{n}{2},$$
where $C_{\beta}=\int_{\R^n} \left( 1 + \left| x \right|^2 \right)^{- \beta}  d x=\pi^{\frac{n}{2}} \frac{\Gamma \left(  \beta - \frac{n}{2} \right)}{\Gamma \left( \beta \right)}$ is a normalizing constant.

In the sequel, we shall adopt the following unified notation. Given $\kappa$ as in~\eqref{defkappa}, we consider a nonnegative function
$W:\R^n\to \R^+$ with the convention that
$$
\begin{cases}
\text{when $\kappa > 0$}, & W \text{ is concave on the bounded open convex set } \{W>0\} \\
\text{when $\kappa < 0$}, & W \text{ is convex on $\R^n$},
\end{cases}
$$
with the property that
$$\int W^{1/ \kappa } < +\infty ;$$
we then define the density 
\begin{equation}\label{defrho}
\rho_{\kappa, W} (x) = \frac1{ \int W^{1/ \kappa}} \, W^{1/\kappa} (x) .
\end{equation}
Our first goal is to study generalized transport inequalities for these probability measures (which we identify with the density). 

Let $\mu$ a probability measure on $\R^n,$ we recall that a transport inequality is an inequality of the form
$$ \alpha \left(  \mathcal{W}_c \left(  \mu , \cdot  \right) \right) \le H \left( \cdot \| \mu \right),$$
where $\alpha$ is an increasing function on $\left[ 0 , +\infty \right)$ with $\alpha \left( 0 \right)=0,$ $\mathcal{W}_c \left( \mu ,  \cdot   \right)$ is the Kantorovich distance from $\mu$ and $H \left( \cdot \| \mu \right)$ a relative entropy with respect to $\mu.$ 
Let us recall that given a cost function $c: \R^n \times \R^n \to \R_+$, the Kantorovich distance $\mathcal W_c \left( \mu , \nu \right)$ between two probability measures $\mu $ and $\nu$ on $\R^n$ is defined by

$$ \mathcal{W}_c \left( \mu , \nu \right)=\inf_{\pi} \iint_{\R^n \times \R^n}  c \left( x , y \right) d \pi \left( x , y \right)$$
where the infimum is taken over all probability measures $\pi$ on $\R^n \times \R^n$ projecting on $\mu$ and $\nu$ respectively. In case where $c \left( x , y \right)=\left| x - y \right|^p$, with $p \ge 1,$ we note

$$\mathcal W_c \left( \mu , \nu \right)=W^p_p \left( \mu , \nu \right).$$
The relative entropy is defined as follows. 

\begin{defi}[Entropy]
Let $\kappa$, $W$ and $\rho_{\kappa , W}$ be given as in the paragraph before~\eqref{defrho}. Given a probability density $\rho$ on $\R^n$ we introduce the $(\kappa,W)$-entropy
$$ H_{\kappa, W} (\rho) := \frac1{\kappa} \int \big(\rho^{1+\kappa} - \rho \big) + \frac{\kappa + 1}{-\kappa} \int \rho \, W $$
provided the integrands are integrable (we set $H_{\kappa, W} (\rho) =+\infty$ otherwise). The relative entropy is then defined by
\begin{eqnarray*}
H_{\kappa, W} (\rho||\rho_{\kappa,W}) &:=& H_{\kappa, W} (\rho)  -H_{\kappa, W} (\rho_{\kappa, W}) \\
&=& \frac{1}{\kappa} \int \left( \rho^{\kappa + 1} - \left( \kappa + 1 \right) \rho \, W \right) + \int W^{1 + 1 / \kappa}.
\end{eqnarray*}
\end{defi}
The reader can convince himself that the functional $\rho\to H_{\kappa, W} (\rho)$ is convex in $\rho$ (note the role played by the sign of $\kappa>-1$) and that $H_{\kappa, W} (\rho||\rho_{\kappa,W}) \ge 0$. Let us emphasize that the log-concave case corresponds to the case $\kappa\to 0$. We can approximate it from above or from below. For instance, given a convex function $V$ with $V\to +\infty$ at infinity,  if we set, for $\kappa < 0$ close to zero with $W \left( x \right) = W_{\kappa} \left( x \right) = \left( 1 - \kappa V \left( x \right) \right)_+$ then, as $\kappa \to 0^-$, 
$$\rho_{\kappa, W_{\kappa} } \to \rho_V:=\frac1{\int e^{-V}} \, e^{-V}$$
and $H_{\kappa, W_\kappa} (\rho||\rho_{\kappa,W_\kappa}) \to \int \log\big(\frac{\rho}{\rho_V}\big) \, \rho$, the classical relative entropy of $\rho$ with respect to $\rho_V$.

Generalized transport inequalities have been studied in~\cite{cegh} in order to study quasilinear parabolic-elliptic equations, under some uniform convexity assumption. The following result can be seen as a dimensional form of the transport inequality for log-concave measures stated in~\cite{ce} that goes back to earlier work by Bobkov and Ledoux. 
The cost is defined by
\begin{equation}\label{defcost}
c_{\kappa, W}(x,y) = \frac{\kappa+1}{-\kappa}  \Big[ W(y) - W(x) - \nabla W(x) \cdot (y-x) \Big].
\end{equation}
According to the context recalled before~\eqref{defrho}, note that $c_{\kappa, W}(x,y) \ge 0$, with $c_{\kappa, W}(x,x) = 0$. This cost is actually mainly independent of $\kappa$, which is there only to distinguish between convex ($\kappa<0$) and concave ($\kappa>0$) situations. 
The general transport inequality is as follows.

\begin{thm}
Let $\kappa$, $W$ and $\rho_{\kappa , W}$ be given as in the paragraph before~\eqref{defrho}. Then we have the following transport inequality, for the entropy and cost defined above: for any probability density $\rho$ on $\R^n$,
\begin{equation}\label{ineqtransport}
W_{c_{\kappa, W}} (\rho_{\kappa , W} , \rho) \le  H_{\kappa, W} (\rho||\rho_{\kappa,W})
\end{equation} 
\end{thm}
According to the discussion above,  when $W(x) = W_{\kappa} (x) = \left( 1 + \kappa V(x) \right)_+$ and $\kappa\to 0^-$, the transport inequality recalled in~\cite{ce} for $\rho_V:=e^{-V}$, namely
$$W_{c_V}(\rho_V, \rho) \le H_V(\rho||\rho_V),$$
 is recovered, for the cost $c_V \left( x , y \right) = V \left( y \right) - V \left( x \right) - \nabla V \left( y \right) \cdot \left( y - x \right)$ and the relative entropy $H_V(\rho||\rho_V)= \int \log \Big(\frac{\rho}{\rho_V} \Big) \rho $. 

The previous inequality is therefore not surprising, and it requires only a minor work to extract it from~\cite{cegh}.

Interestingly enough, we will show that the previous inequality allows to reproduce, by a linearization procedure, the dimensional Brascamp-Lieb inequalities obtained by Bobkov and Ledoux~\cite{ble0} and Nguyen~\cite{n} thus providing a mass transport approach to them. 

Our second goal is to obtain quantitative versions of the transport inequality above.  Before announcing our results, we need some notation. Let the function $\mathcal F$ be defined on $\R_+$ by

$$\mathcal F \left( t \right) := t - \log \left( 1 + t \right), \qquad \forall t \ge 0.$$
The function $\mathcal F$ is an increasing, convex function on $\R_+$ and it behaves like $t^2$ when $t$ is small and like $t$ when $t$ is large, more precisely:

$$\frac14 \min \left\{ t , t^2 \right\} \le \mathcal F \left( t \right) \le \min \left\{ t , t^2 \right\}, \qquad \forall t \ge 0.$$

We introduce next a (weighted) isoperimetric type constant.  Given a probability measure $\mu$, we denote by  $h_{W}(\mu)$  the best nonnegative constant such that the following inequality 
\begin{equation}\label{weightedpoincare}
\int \mathcal F \left( \left| \nabla f \right| \right) W d \mu \ge\int \mathcal F \left(  h_W (\mu)   \left|  f - m_f \right| \right) d \mu 
\end{equation}
holds for every smooth enough function $f \in L^1 \left( \mu \right)$. One may hope that $h_W (\rho_{\kappa, W})>0$. We shall briefly discuss this in the last section.

It may be convenient to change the notation and focus rather on the parameter
$$\beta = \pm \frac1\kappa$$
according to the {\bf Case 1} and {\bf Case 2} detailed previously. With some abuse of notation, we will denote $\rho_{\beta, W}$, $H_{\beta, W}$ and $c_{\beta, W}$ the corresponding quantities. 

So, more explicitly, in {\bf Case 1}, which corresponds to $\kappa>0$,  we are given a $\beta \in [0,+\infty[$ and function $W:\R^n \to \R^+$
concave on its support such that  
$$\rho_{\beta, W}(x):= \frac{W(x)^\beta}{\int W^\beta}$$
is a probability density. In this {\bf Case 1} the cost is $c_{\beta , W} \left( x ,  y \right) \allowbreak=\allowbreak  \left( \beta + 1 \right) \big( W \left( x \right) \allowbreak - W \left( y \right) \allowbreak + \nabla W \left( x \right) \cdot \left( y - x \right) \big)$
and the relative entropy $ H_{\beta , W} (\rho \| \rho_{\beta , W} ) := \int \left( \beta \rho^{1 + 1 / \beta} - \left( \beta + 1 \right) \rho W  \right)  + \int W^{\beta + 1} $.

\begin{thm}
Under the notation of {\bf Case 1} recalled above, introduce the costs 
$$\widetilde{c} \left( x , y \right)=c   \mathcal F \left(h_W(\rho_{\beta , W}) \, \left| y - x \right|  \right)$$ 
where $c>0$ is some fixed numerical constant, and 
$$c \left( x , y \right)=c_{\beta,W} \left( x ,  y \right) + \widetilde{c} \left( x , y \right).$$
Then, if  $\rho$ and $\rho_{\beta, W}$  have the same center of mass, we have
\begin{equation}
H_{\beta, W} (\rho ||\rho_{\beta , W}) \ge \mathcal W_{c} \left( \rho_{\beta , W}, \rho \right).
\label{transport_1'''''}
\end{equation}
\end{thm}

\begin{rmq}

Since the inequality $\mathcal W_c \left( \rho , \rho_{\beta , W} \right) = \mathcal W_{ c_{\beta, W} + \widetilde{c}  }  \left( \rho , \rho_{\beta , W} \right) \ge \mathcal W_{ c_{\beta, W} }  \left( \rho , \rho_{\beta , W} \right) + \mathcal W_{ \widetilde{c} }  \left( \rho , \rho_{\beta , W} \right)$ holds, the second transport inequality gives a remainder term for the first inequality:

\begin{equation}
H_{\beta , W} \left( \rho \| \rho_{\beta, W} \right) - \mathcal W_{ c_{\beta, W} } \left( \rho , \rho_{\beta , W} \right) \ge \mathcal W_{ \widetilde{c} }  \left( \rho , \rho_{\beta , W} \right).
\label{linear1}
\end{equation}

\end{rmq}
In {\bf Case 2}, which corresponds to $\kappa \in \left[ - \frac{1}{n} , 0 \right[$,  we are given a $\beta \ge n$ and function $W:\R^n \to \R^+$
convex such that  
$$\rho_{\beta, W}(x):= \frac{W(x)^{-\beta } }{\int W^{- \beta}  }$$
is a probability density. In this {\bf Case 2} the cost is $c_{\beta , W} \left( x ,  y \right) \allowbreak=\allowbreak  \left( \beta - 1 \right) \big( W \left( y \right) \allowbreak - W \left( x \right) \allowbreak - \nabla W \left( x \right) \cdot \left( y - x \right) \big)$
and the relative entropy $ H_{\beta , W} (\rho \| \rho_{\beta, W} ) := \int \left(  \left( \beta - 1 \right) \rho W - \beta \rho^{1 - 1 / \beta} \right) + \int W^{1 - \beta}  .$

\begin{thm}

Under the notation of {\bf Case 2} recalled above, introduce the costs 
$$\widetilde{c} \left( x , y \right)=  \frac{c}{\beta}  \left( 1 - \frac{n}{\beta} \right)^2   \mathcal F \left( h_W (\rho_{\beta , W}) \, \left| y - x \right|  \right)$$ 
where $c>0$ is some fixed numerical constant, and 
$$c \left( x , y \right)=c_{\beta,W} \left( x ,  y \right) + \widetilde{c} \left( x , y \right).$$
Then, if  $\rho$ and $\rho_{\beta, W}$  have the same center of mass, we have
\begin{equation}
H_{\beta, W} (\rho ||\rho_{\beta , W}) \ge \mathcal W_{c} \left( \rho_{\beta , W}, \rho \right).
\label{transport_2'''''}
\end{equation}

\end{thm}

\begin{rmq}
As in the {\bf Case 1}, this gives a remainder term for the first transport inequality:

\begin{equation}
H_{\beta} \left( \rho \right) - \mathcal W_{c_{\beta, W}} \left( \rho , \rho_{\beta , W} \right) \ge \mathcal W_{\widetilde{c}} \left( \rho , \rho_{\beta , W} \right).
\label{linear2}
\end{equation}

\end{rmq}
The idea of the proof is to transport the densities $\rho$ onto the measure $\rho_{\beta , W}.$ Cordero in~\cite{ce} uses optimal transportation to obtain a transport inequality for log-concave measures. We recall some backgrounds about mass transportation at the beginning of the following section but we refer to~\cite{vi} for a detailed approach.

In a second section, we will use transport inequalities to retrieve some dimensional versions of Brascamp-Lieb inequalities. Such inequalities had already been studied by Bobkov and Ledoux in~\cite{ble0} where they use a Pr\'{e}kopa-Leindler type inequality. More recently, Nguyen in~\cite{n} retrieve these inequalities with a $L^2-$ H\"{o}rmander method. Our approach is different. From transport inequalities (Theorems 1 and 2), we will use a linearization procedure to retrieve these inequalities. 
\newline
\newline
I would like to thank my Professor Dario Cordero-Erausquin for his encouragements, his careful reviews and his many useful discussions.

\section{Proof of Theorem 1}

In this part, we do not use the notation $\beta$ because there it is useless to separate the proof between {\bf Case 1} and {\bf Case 2}. We can assume that $\int W^{1/ \kappa} = 1$. The proof is based on optimal transportation. Let us recall briefly what it is about. Let two probability measures $\mu$ and $\nu$ on $\R^n.$ We say a map $T : \R^n \to \R^n$ transports the measure $\mu$ onto the measure $\nu$ if:

$$\nu \left( B \right)= \mu \left( T^{-1} \left( B \right) \right), \qquad \text{for all borelian sets} \, B \subseteq \R^n.$$
This gives a transport equation: for all nonnegative Borel function $b: \R^n \to \R_+,$ 

\begin{equation}
\int_{\R^n} b \left( y \right)  d \nu \left( y \right) = \int_{\R^n} b \left( T \left( x \right) \right) d \mu \left( x \right).
\label{trans1}
\end{equation}
When $\mu$ and $\nu$ have densities with respect to Lebesgue measure (it will be always the case in this paper), say $F$ and $G$,~\eqref{trans1} becomes:

\begin{equation}
\int_{\R^n} b \left( y \right) G \left( y \right)  d y= \int_{\R^n} b \left( T \left( x \right) \right) F \left( x \right)  d x.
\label{trans2}
\end{equation}
The existence of a such map $T$ is resolved by the following Theorem of Brenier~\cite{br} and refined by McCann~\cite{mc1}.

\begin{thm}

If $\mu$ and $\nu$ are two probability measures on $\R^n$ and $\mu$ is absolutely continuous with respect to Lebesgue measure, then there exists a convex function $\varphi$ defined on $\R^n$ such that $\nabla \varphi$ transports $\mu$ onto $\nu.$ Furthermore, $\nabla \varphi$ is uniquely determined $\mu$ almost-everywhere.

\end{thm}
As $\varphi$ is convex on its domain, it is differentiable $\mu$ almost-everywhere. If we assume $\varphi$ of class $C^2,$ the change of variables $y=\nabla \varphi \left( x \right)$ in~\eqref{trans2} shows that $\varphi$ satisfies the Monge-Amp\`{e}re equation, for $\mu$ almost-every $x \in \R^n:$

\begin{equation}
F \left( x \right)= G \left(  \nabla \varphi \left( x \right) \right) \det D^2 \varphi \left( x \right).
\label{idk}
\end{equation}
Here $D^2 \varphi \left( x \right)$ stands for the Hessian matrix of $\varphi$ at the point $x.$ Cafarelli's Theorems~\cite{ca1} and~\cite{ca2} asserts the validity of~\eqref{idk} in classical sense when the functions $F$ and $G$ are H\"{o}lder-continuous and strictly positive on their respective supports. Generally speaking, the matrix $D^2 \varphi \left( x \right)$ can be defined with the Taylor expansion of $\varphi$ ($\mu$ almost-everywhere)
$$\varphi \left( x + h \right) =_{h \to 0} \varphi \left( x \right) + \nabla \varphi \left( x \right) \cdot h + \frac12  D^2 \varphi \left( x \right) \left( h \right) \cdot h + o \left( \left| h \right|^2 \right).$$
In our case we are given a probability density $\rho$ on $\R^n$, which we can assume to be, by approximation,  continuous and strictly positive. Let $T=\nabla \varphi$ the Brenier map between $\rho_{\kappa,W}$ and $\rho$.  Because $\rho_{\kappa,W}$ has a convex support, and is continuous on its support, we know that $\varphi \in W^{2,1}_{\mathrm{loc}}$. Then the following integration by parts formula

$$\int f \,  \Delta \varphi = - \int \nabla \varphi \cdot \nabla f$$
is valid for any smooth enough function $f: \Omega \to \R.$ We begin by writing Monge-Amp\`{e}re equation:
\begin{equation}
\rho_{\kappa, V} \left( x \right)= \rho \left( T \left( x \right) \right) \det  D^2 \varphi.
\label{MA_conc}
\end{equation}
It follows that for  
\begin{equation}
\rho \left( T \left( x \right) \right)^{\kappa} = \rho_{\kappa ,W} \left( x \right)^{\kappa}  \left(  \det D^2 \varphi  \right)^{-\kappa}
=   W(x) \left( \det D^2 \varphi \right)^{-\kappa}
\label{MA_conc2}
\end{equation}
Recall that for $\kappa \in [-\frac1n, +\infty]$, the functional
$$M\to  \frac1\kappa {\det}^{-\kappa}(M)$$
is concave on the set of nonnegative symmetric $n\times n$ matrices. If we consider the tangent at the identity matrix $I$ we find that
$$ \frac1\kappa {\det}^{-\kappa}(M) \ge  \frac1\kappa - \mathrm{tr}(M-I).$$
Actually, for future use, let us 
introduce
\begin{equation}
\mathcal G_{\kappa}(M):= \frac1\kappa {\det}^{-\kappa}(M) -  \frac1\kappa + \textrm{tr}(M-I)\ge 0.
\label{mainiag}
\end{equation}
So if we introduce the displacement function function $\theta(x) = \varphi(x) - |x|^2/2 $ so that $T (x ) = \nabla\varphi(x) = x + \nabla\theta(x)$ we have
$$\frac1\kappa\rho \left( T \left( x \right) \right)^{\kappa} \ge  W(x) \left( \frac1\kappa - \Delta \theta(x) \right) + W(x) \mathcal G_{\kappa}(D^2\theta(x)) $$
Integrating with respect to $\rho_{\kappa, W}=W^{1/\kappa}$ and performing an integration by parts (note that $W^{1+\frac1\kappa} \to 0 $ at infinity) we find
\begin{align*}\frac1\kappa \int \rho^{1+\kappa} &\ge  \frac1\kappa \int \rho_{\kappa, W}^{1+\kappa} - \int W^{1+\frac1\kappa}\Delta \theta 
+ \int  W \mathcal G_{\kappa}(D^2\theta(x)) \rho_{\kappa, W}\\
&=\frac1\kappa \int \rho_{\kappa, W}^{1+\kappa} + \frac{1+\kappa}{\kappa} \int W^{\frac1\kappa}\nabla W \cdot \nabla \theta 
+ \int  W \mathcal G_{\kappa} (D^2\theta(x)) \rho_{\kappa, W}
\end{align*}
By definition of mass transport we have 
$$ \frac{1+\kappa}{\kappa}  \int W(y) \, \rho(y) \, dy =\frac{1+\kappa}{\kappa}  \int W(T(x)) W^{\frac1\kappa}(x) \, dx $$
so adding the left-hand expression to the left and the right-hand expression to the right we find (adding also the required cosmetic constant) that
$$ H_{\kappa, W}(\rho) =  H_{\kappa , W}(\rho_{\kappa, W}) + \int c_{\kappa, W}(x, T(x)) \, \rho_{\kappa, W} +  \int  W \mathcal G_{\kappa} (D^2\theta(x)) \rho_{\kappa, W},$$
or equivalently
\begin{equation}\label{mainstep}
H_{\kappa, W}( \rho || \rho_{\kappa, W} ) \ge \int c_{\kappa, W}(x, T(x)) \, \rho_{\kappa, W} +  \int  W \mathcal G_{\kappa} (D^2\theta(x)) \rho_{\kappa, W}.
\end{equation}
In particular, since $\mathcal G_\kappa \ge 0$, we find, by the definition of the transportation cost, the inequality stated in Theorem 1.

\section{Remainder terms (Theorems 2 and 3)}

The main step is to obtain a quantitative form of the inequality~\eqref{mainiag} and the approach is not the same whether we are in {\bf Case 1} or in {\bf Case 2}. The rest of the proof is exactly the same.

\subsection{Case 1}

We start from~\eqref{mainstep} and try to exploit the last term in order to get an improved inequality. The following Lemma gives a quantitative form of the inequality~\eqref{mainiag}. 

\begin{lem}

Under the notation of {\bf Case 1}, for any symmetric $n \times n$ matrix $M$, we have

\begin{equation}
\mathcal G_{\kappa} \left( M \right) \ge c \sum_{i=1}^n \min \left\{ \mu_i^2 , \left| \mu_i \right| \right\},
\label{iag_2}
\end{equation}
where $\mu_1, \cdots , \mu_n$ are the eigenvalues of $M - I$ and for some numerical constant $c>0.$

\end{lem}

\begin{proof}
The main point is the following inequality valid for $t \ge -1,$

$$  \log \left( 1 + t \right) \le t - c \min \left\{ t^2 , \left| t \right| \right\}  ,$$
where $c > 0$ is a numerical constant (for instance $c=\frac{3}{10}$ works). Then, applying it with $\mu_i$ and after summing, this gives

$$\sum_{i=1}^n \log \left( 1 + \mu_i \right) \le \sum_{i=1}^n \mu_i - c \sum_{i=1}^n \min \left\{ \left| \mu_i \right| ,  \mu_i^2  \right\}, $$
and 
\begin{eqnarray*}
\prod_{i=1}^n \left( 1 + \mu_i \right)^{- 1 / \beta} &\ge& \exp \left( -\frac{1}{\beta} \sum_{i=1}^n \mu_i + c \frac{1}{\beta} \sum_{i=1}^n \min \left\{ \left| \mu_i \right| ,  \mu_i^2  \right\} \right) \\
&\ge& 1 - \frac{1}{\beta} \sum_{i=1}^n \mu_i + c \frac{1}{\beta} \sum_{i=1}^n \min \left\{ \left| \mu_i \right| ,  \mu_i^2  \right\}. 
\end{eqnarray*}
Since $\prod_{i=1}^n \left( 1 + \mu_i \right)^{-1 / \beta }= \det^{- 1 / \beta} \left( M \right)$ and $\sum_{i=1}^n \mu_i= \mathrm{tr } \left( M - I \right)$ dividing by $\frac{1}{\beta} > 0$ ends the proof. 

\end{proof}
Let us prove now Theorem 2. 

\begin{proof}
We go back to~\eqref{mainstep} and we use the previous Lemma to minimize $\int W \mathcal G_{\kappa} \left( D^2 \theta \left( x \right) \right) \rho_{\beta , W}:$ 

\begin{eqnarray*}
\int W \mathcal G_{\kappa} \left( D^2 \theta \left( x \right) \right) \rho_{\beta, W} \ge c \int \mathrm{tr} \left( \mathcal F \left( D^2 \theta \left( x \right) \right) \right)  W \rho_{\beta, W}
\end{eqnarray*}
Now, we follow the approach of Cordero-Erausquin in~\cite{ce}.

\begin{lem}{\cite{ce}} For any $n \times n$ symmetric matrix $M$ with eigenvalues larger than $-1,$ we have:

$$\mathrm{tr} \left(  \mathcal F \left( M \right) \right)  \ge \frac18 \int_{\mathbb{S}^{n-1}} \mathcal F \left( \sqrt{n} \left| M u \right| \right)  d \sigma \left( u \right).$$

\end{lem}
This gives

\begin{equation}
\int W \mathcal G_{\kappa} \left( D^2 \theta \left( x \right) \right) \rho_{\beta , W} \ge c \int_{\mathbb S^{n-1}} \left( \int  \mathcal F \left( \sqrt{n } \left| D^2 \theta \left( x \right) u \right| \right) W \rho_{\beta, W}  \right) d \sigma \left( u \right)
\label{aaa}
\end{equation}
Since $D^2 \theta \left( x \right) u =  \nabla \left( \nabla \theta \left( x \right) \cdot u \right),$ and using~\eqref{weightedpoincare} in~\eqref{aaa}, we find

\begin{equation}
\int W \mathcal G_{\kappa} \left( D^2 \theta \left( x \right) \right) \rho_{\beta , W} \ge c \int_{\mathbb S^{n-1}} \int \mathcal F   \left(  h_W \left(  \rho_{\beta , W} \right) \sqrt{n} \left|   \nabla \theta \cdot u \right|  \right)      \rho_{\beta , W} d \sigma \left( u \right)
\label{eq1}
\end{equation}
Note that since $\rho$ and $\rho_{\beta,W}$ have the same center of mass, we have 
$$\int \nabla \theta \left( x \right) \cdot u \, \rho_{\beta ,  W}=0.$$
Before going on, let us use the following Fact. 

\begin{fact}

There exists $c_n > 0,$ such that for all $x \in \R^n,$ we have

$$\int_{\mathbb{S}^{n-1}} \left| x \cdot u \right|  d \sigma \left( u \right)= c_n \left| x \right|.$$
Moreover, one can prove that there exists two positive numerical constants, say $c$ and $C,$ such that $c  \le c_n \sqrt{n} \le C.$ 

\end{fact}

\begin{proof}

It is easy to see that $N \left( x \right):=\int_{\mathbb{S}^{n-1}} \left| x \cdot u \right|  d \sigma \left( u \right)$ is a norm invariant with rotations, then it is a multiple of the Euclidean norm. It is classical, see~\cite{bor1, bor2}, that $c_n \simeq \sqrt{\int_{\mathbb S^{n-1}}   \left| x \cdot u \right|^2 d \sigma \left( u \right)}$ (i.e. up to numerical constants). Then, one can prove, using concentration of measures, that

$$\sqrt{\int_{\mathbb S^{n-1}}   \left| x \cdot u \right|^2 d \sigma \left( u \right)} \simeq \frac{1}{\sqrt{n}}.$$

\end{proof}
Using Fubini's theorem, Jensen's inequality ($\mathcal F$ is convex) and Fact 1 in~\eqref{eq1}, we find

\begin{eqnarray*}
\int W \mathcal G_{\kappa} \left( D^2 \theta \left( x \right) \right) \rho_{\beta , W} &\ge& c \int \mathcal F \left(  h_{W} \left( \rho_{\beta , W} \right)    \sqrt{n} \int_{\mathbb{S}^{n-1}} \left|  \nabla \theta \left( x \right) \cdot u  \right|  d \sigma \left( u \right) \right)  \rho _{\beta , W} \\
&\ge& C \int  \mathcal F \left(  h_{W} \left( \rho_{\kappa , W} \right)   \left|  \nabla \theta \left( x \right)  \right|  \right)  \rho _{\beta , W} \\
&=& C \int \widetilde{c} \left( x , T \left( x \right) \right) \rho_{\beta , W}
\end{eqnarray*}
Replacing this inequality in~\eqref{mainstep} finishes the proof of Theorem 2.

\end{proof}

\subsection{Case 2}

As we say at the beginning of this part, the proof of Theorem 3 is very similar as the one for Theorem 2, the only difference is proof of the quantitative form of~\eqref{mainiag}. That is the goal of the following Lemma.

\begin{lem}

Under the notation of {\bf Case 2}, for any nonnegative, symmetric $n \times n $ matrix $M$ and for all $\beta \ge n,$ we have:

\begin{equation}
\mathcal G_{\kappa} \left( M \right)  \ge  \frac{3}{64 \beta} \left( 1 - \frac{n}{\beta} \right)^2 \mathcal F \left(  \| M - I \|_{\mathrm{HS}} \right).
\label{lemma4}
\end{equation}

\end{lem}

\begin{proof}

We introduce the probability measure $\mu$ defined on $\R$ by $ d \mu=\frac1\beta \delta_1 + \cdots + \frac1\beta \delta_n + \left( 1 - \frac{n}{\beta} \right) \delta_{n+1},$ the function $\phi$ defined on $\left[ - 1 , +\infty \right)$ by $\phi \left( x \right)=\log \left( 1 + x \right)$ and the function $f$ defined on $\R$ by:

$$f(x)=\begin{cases} \mu_i \,  \text{if } x=i \, \text{with} \, i \in \left\{ 1 , \cdots , n  \right\}, \\ 0 \, \text{else}. \end{cases}$$
Let us note that $\phi \left( \int_{\R} f d \mu \right)=\log \left( 1 + \frac1\beta \sum_{i=1}^n \mu_i \right)$ and $\int_{\R} \phi \left( f \right) d \mu=\log \left( \prod_{i=1}^n \left( 1  + \mu_i \right)^{\frac1\beta}  \right).$ We start with this inequality: for all $s,t \in \R_+,$ 

\begin{equation}
\log \left( s \right) \le \log \left( t \right) + \frac{s-t}{t} - \frac{\left( s - t \right)^2}{2 \max \left\{ s , t \right\}^2}.
\label{logre}
\end{equation}
In~\eqref{logre}, taking $s=1 + f$ and $t=1+m=1+ \int_{\R} f d \mu, $ then integrating with respect to the measure $\mu,$ it gives:

$$  \int_{\R} \phi \left( f \right)  d \mu \le \phi \left( m \right) - \frac{1}{2} \int_{\R}  \frac{ \left( f - m \right)^2 }{ \max \left\{  1+m , 1 + f \right\}^2 }  d \mu .$$
Then, we have:

\begin{eqnarray*}
\phi \left( m \right) - \int_{\R} \phi \left( f \right) d \mu &\ge& \frac{1}{2} \int_{\R}  \frac{ \left( f - m \right)^2 }{ \max \left\{  1+m , 1 + f \right\}^2 }  d \mu \\
&\ge& \frac{1}{4  \left( 1+ \mu^2_{\mathrm{max}} \right) }  \int_{\R} \left( f - m \right)^2 d \mu.
\end{eqnarray*}
Let us compute $\int_{\R} \left( f - m \right)^2 d \mu.$

\begin{eqnarray*}
\int_{\R} \left( f - m \right)^2  d \mu &=& \int_{\R} f^2  d \mu - \left( \int_{\R} f d \mu \right)^2 \\
&=& \frac1\beta \sum_{i=1}^n \mu_i^2 - \frac{1}{\beta^2} \left( \sum_{i=1}^n \mu_i \right)^2 \\
&\underbrace{\ge}_{\text{Cauchy-Schwarz inequality}}& \frac{1}{\beta} \left( 1 - \frac{n}{\beta} \right) \sum_{i=1}^n \mu_i^2 \\
&=& \frac{1}{\beta} \left( 1 - \frac{n}{\beta} \right) \| M - I \|_{\mathrm{HS}}^2.
\end{eqnarray*}
So, we have:

$$ \phi \left( m \right) - \int_{\R} \phi \left( f \right) d \mu \ge \frac{1}{4\beta} \left( 1 - \frac{n}{\beta} \right) \frac{\| M - I \|_{\mathrm{HS}}^2}{1+ \mu^2_{\mathrm{max}}}=:z.$$
Taking the exponential, this yields $e^{ \int_{\R} \phi \left( f \right)  d \mu } \le e^{-z} e^{  \phi \left( m \right) }$ then $e^{ \phi \left( m \right) } - e^{ \int_{\R}  \phi \left( f \right)  d \mu } \ge \left( 1 - e^{-z} \right) e^{ \phi \left( m \right) }.$ It is easy to see that $z \in \left[ 0 , \frac12 \right],$ so the inequality $1 - e^{-z} \ge \frac34 z$ holds.
Finally, we have established the following inequality:

$$1 + \mathrm{tr} \left( M - I \right) - \det \left( M \right)^{\frac{1}{\beta}} \ge \frac{3}{16 \beta} \left(1 - \frac{n}{\beta} \right) \frac{\| M - I \|_{\mathrm{HS}}^2}{1+ \mu^2_{\mathrm{max}}} \left( 1 + \frac1\beta \mathrm{tr} \left( M - I \right) \right) .$$
To conclude, we discuss whether $\mu_{\mathrm{max}}$ is bigger than $1$ or not. \\
\emph{We assume $\mu_{\mathrm{max}} \le 1.$} In this case, we have:

$$1 + \mathrm{tr} \left( M - I \right) - \det \left( M \right)^{\frac{1}{\beta}} \ge \frac{3}{32 \beta} \left(1 - \frac{n}{\beta} \right)^2 \| M - I \|_{\mathrm{HS}}^2.$$
\emph{We assume $\mu_{\mathrm{max}} \ge 1.$} First, we work on $1 + \frac1\beta \mathrm{tr} \left( M - I \right).$ This yields the following lines:

\begin{eqnarray*}
\beta \left( 1 + \frac1\beta \mathrm{tr} \left( M - I \right) \right) &=& \beta + \sum_{i=1}^n \mu_i \\
&\ge&  \left( \beta - \left( n - 1 \right) \right) + \mu_{\mathrm{max}} \\
&\ge&  \mu_{\mathrm{max}} \\
&\ge& \frac1n \sum_{i=1}^n \left| \mu_i \right| \\
&\ge& \frac{1}{n}  \sqrt{\sum_{i=1}^n \mu_i^2} \\
&=& \frac1n \| M - I \|_{\mathrm{HS}}.
\end{eqnarray*}
Consequently,

$$ 1 + \mathrm{tr} \left( M - I \right) - \det \left( M \right)^{\frac{1}{\beta}} \ge \frac{3}{16 n \beta^2} \left(1 - \frac{n}{\beta} \right) \| M-I \|_{\mathrm{HS}} \frac{\| M - I \|_{\mathrm{HS}}^2}{1+ \mu^2_{\mathrm{max}}}.$$
We finally conclude thanks to 

$$\frac{\| M - I \|_{\mathrm{HS}}^2}{1+ \mu^2_{\mathrm{max}}} \ge \frac{n \mu^2_{\mathrm{max}}}{ 1+ \mu^2_{\mathrm{max}} }\ge \frac12 n$$
and 

$$1 + \mathrm{tr} \left( M - I \right) - \det \left( M \right)^{\frac{1}{\beta}} \ge \frac{3}{32 \beta^2} \left(1 - \frac{n}{\beta} \right) \| M-I \|_{\mathrm{HS}}.$$
In the two cases, we have at the same time the inequality:

\begin{eqnarray*}
1 + \mathrm{tr} \left( M - I \right) - \det \left( M \right)^{\frac{1}{\beta}} &\ge& \frac{3}{32 \beta^2} \left( 1 - \frac{n}{\beta} \right)^2 \min \left\{ \| M - I \|_{\mathrm{HS}} ,  \| M - I \|_{\mathrm{HS}}^2  \right\} \\
&\ge& \frac{3}{32 \beta^2} \left( 1 - \frac{n}{\beta} \right)^2 \mathcal F \left(  \| M - I \|_{\mathrm{HS}} \right).
\end{eqnarray*}
Multiplying by $\beta > 0,$ this concludes the proof of the Lemma.

\end{proof}
Let us end the proof of Theorem 3.

\begin{proof}
Let us plug~\eqref{lemma4} in~\eqref{mainstep}, we obtain,

$$\int  W \mathcal G_{\kappa} (D^2\theta(x)) \rho_{\beta , W} \ge   \frac{3}{64 \beta}  \left( 1 - \frac{n}{\beta} \right)^2 \int \mathcal F \left(  \| M - I \|_{\mathrm{HS}} \right) W \rho_{\beta , W}.$$
The rest of the proof is the same as the one for Theorem 2.

\end{proof}

\section{Linearization and dimensional Brascamp-Lieb inequalities}

\subsection{Dimensional Brascamp-Lieb inequalities}

The goal of this part is to recover dimensional Brascamp-Lieb inequalities. For that, we linearize our transport inequality we established in Theorem 1. Let us cite the result we need for that.

\begin{lem}{\cite{ce}}

Let $c: \R^n \times \R^n \to \R_+$ a function such that $c \left( y , y \right)=0$ and $c \left( x  , y \right) \ge \delta_0 \left| x - y \right|^2$ for all $x , y \in \R^n$ and for some $\delta_0 > 0.$ We assume that for every $y \in \R^n,$ there exists a nonnegative, symmetric matrix $n \times n$, say $H_y,$ such that

$$c \left( y , y + h \right) =_{h \to 0} \frac12 H_y h \cdot h + \left| h \right|^2 o \left( 1 \right).$$
Then, if $\mu$ is a probability measure on $\R^n$ and $g$ is $C^1$ compactly supported with $\int_{\R^n} g \,  d \mu=0,$ we have

$$\liminf_{\epsilon \to 0} \frac{1}{\epsilon^2} \mathcal{W}_c \left( \mu , \left( 1 + \epsilon g \right) \mu \right) \ge \frac12 \frac{\left( \int_{\R^n} g \, f \,  d\mu\right)^2}{ \int_{\R^n}  H_y^{-1} \nabla f \cdot \nabla f  d \mu },$$
for any function $C^1$ compactly supported $f.$

\end{lem}

Using this Lemma to linearize inequality~\eqref{ineqtransport} gives

\begin{thm}
With the notation of Theorem 1 and assuming $\int W^{1 / \kappa}=1,$ we have the following inequality

\begin{equation}
- \kappa \int \left( D^2 W  \right)^{-1} \nabla f \cdot \nabla f \, \rho_{\kappa , W} \ge \int g^2 \, W \, \rho_{\kappa , W},
\label{transportlinea}
\end{equation}
with $\int g \,  \rho_{\kappa , W}=0$ and $f=g \, W.$

\end{thm}

As we said in Introduction, we can now retrieve dimensional Brascamp-Lieb inequalities. For example, in {\bf Case 1},~\eqref{transportlinea} becomes

$$  \int \left( - D^2 W  \right)^{-1} \nabla f \cdot \nabla f \, \rho_{\beta , W} \ge \beta \int g^2 \, W \, \rho_{\beta , W},$$
whereas in {\bf Case 2}
$$ \int \left( D^2 W  \right)^{-1} \nabla f \cdot \nabla f \, \rho_{\beta , W} \ge \beta \int g^2 \, W \, \rho_{\beta , W}.$$

\begin{proof}
Let us remark first that, when $h \to 0,$

\begin{equation}
c \left( y , y + h \right) = \frac12 \left(  \frac{ \kappa + 1}{- \kappa} D^2 W \left( y \right) \right) \left( h \right) \cdot h + \left| h \right|^2 o \left( 1 \right).
\label{schnitzler}
\end{equation}
Let us compute, for $g$ verifying $\int g \, \rho_{\kappa , W} = 0$

$$\liminf_{\epsilon \to 0} \frac{1}{\epsilon^2} H_{\kappa, W} \left( \left( 1 + \epsilon g \right) \rho_{\kappa , W} \| \rho_{\kappa , W} \right). $$
Thanks to Theorem 1, it we will have a maximization of 

$$\liminf_{\epsilon \to 0} \frac{1}{\epsilon^2} W_{c_{\kappa , W}} \left( \rho_{\kappa , W} ,  \left( 1 + \epsilon g \right) \rho_{\kappa , W}   \right).$$ 
Using the definition of the entropy, we have 

\begin{equation}
H_{\kappa , W} \left(  \left( 1 + \epsilon g \right) \rho_{\kappa , W} \| \rho_{\kappa , W}  \right) = \frac{\kappa + 1}{2} \epsilon^2 \int g^2 \, \rho_{\kappa, W}^{1 + \kappa} + o \left( \epsilon ^2 \right).
\label{marai}
\end{equation}
Putting together~\eqref{schnitzler} and~\eqref{marai} thanks to the above Lemma gives the following inequality, for all function $f$ $C^1$ compactly supported

$$\frac{\kappa + 1}{2} \int g^2 \, W \rho_{\kappa , W} \ge \frac12 \frac{     \left(  \int   g \, f  \, \rho_{\kappa , W} \right)^2    }{  \int  \left(  \frac{\kappa + 1}{ - \kappa} D^2 W \right)^{-1} \nabla f \cdot \nabla f \, \rho_{\kappa , W}  }.$$
Taking $f=g W$ concludes the proof of the Theorem.

\end{proof}

\subsection{Quantitative forms}

In this section, we are interested by giving some quantitative forms of the inequalities stated in~\eqref{transport_1'''''} and~\eqref{transport_2'''''}. The main argument is, once again, Lemma 4: we use it with the costs we introduced in Theorems 2 and 3. We separate our result whether we are in {\bf Case 1} or in {\bf Case 2}.

\begin{thm}
Under the notation of {\bf Case 1}, we have the following inequality

$$\int \left( - D^2 W + \frac{c}{\beta + 1} h_W \left( \rho_{\beta , W} \right) I \right)^{-1} \nabla f \cdot  \nabla f \, \rho_{\beta , W} \ge \beta \int g^2 \, W \, \rho_{\beta  ,  W} ,$$
for some numerical constant $c> 0$ and with $\int g \,  \rho_{\beta , W}=0,$ $\int x g \left( x \right) \rho_{\beta , W} \left( x \right) = 0$ and $f=g \, W.$

\end{thm}
And
\begin{thm}

Under the notation of {\bf Case 2}, we have the following inequality

$$\int \left( D^2 W + \frac{c}{\beta \left( \beta - 1 \right)} \left( 1 - \frac{n}{\beta} \right)^2 h_W \left( \rho_{\beta , W} \right) I   \right)^{-1} \nabla f \cdot  \nabla f \, \rho_{\beta , W} \ge \beta \int g^2 \, W \, \rho_{\beta,W},$$
for some numerical constant $c> 0$ and with $\int g \,  \rho_{\beta , W}=,0$ $\int x g \left( x \right) \rho_{\beta , W} \left( x \right) = 0$ and $f=g \, W.$

\end{thm}

The proofs are very similar as the one for Theorem 5. Anyway, let us proof Theorem 6.

\begin{proof}

We keep the notation of Theorem 2. As $\int g \,  \rho_{\beta , W}=0$ and $\int x g \left( x \right) \rho_{\beta , W} \left( x \right) = 0,$ the measures $\rho_{\beta , W}$ and $\left( 1 + \epsilon g \right) \rho_{\beta, W}$ are both probability measures with the same center of mass. Thanks to Theorem 2, it is enough to give as estimation of the relative entropy instead of $\mathcal W_c \left( \rho_{\beta , W} , \left( 1 + \epsilon g \right) \rho_{\beta , W} \right).$ Proof of Theorem 5 gives for the relative entropy:

\begin{eqnarray*}
H_{\kappa , W} \left(  \left( 1 + \epsilon g \right) \rho_{\kappa , W} \| \rho_{\kappa , W}  \right) &=& \frac{\kappa + 1}{2} \epsilon^2 \int g^2 \, \rho_{\kappa, W}^{1 + \kappa} + o \left( \epsilon ^2 \right) \\
&=& \frac{\beta + 1}{2 \beta}  \epsilon^2 \int g^2 \, \rho_{\kappa, W}^{1 + \kappa} + o \left( \epsilon ^2 \right).
\end{eqnarray*}
Thanks to the definition of $\mathcal F,$ one have

$$\lim_{h \to 0 } c \left( y , y +h \right) = \frac12 \left( - \left( \beta + 1 \right) D^2 W \left( y \right) + c h_W \left( \rho_{\beta , W} \right) I \right),$$
for some numerical constant $c > 0.$ Using Lemma 4 with $f = g \, W$ permits to conclude the proof.

\end{proof}

\section{Further remarks on weighted Poincar\'{e} inequalities}

\subsection{Generality on weighted Poincar\'{e} inequalities}

In~\eqref{weightedpoincare}, we introduced $h_W \left( \mu \right)$ as the best nonnegative constant such that the inequality

$$  \int \mathcal F \left( \left|  \nabla f  \right| \right) W d \mu \ge  \int  \mathcal F \left( h_W \left( \mu \right) \left| f - m_f \right| \right) d \mu$$
holds for every smooth enough $f \in L^1 \left( \mu \right).$ Nevertheless, we are convinced that the following definition for weighted Poincar\'e inequality (note that the weight has not the same place):

\begin{equation}
\int \mathcal F \left( \frac{1}{h_W \left( \mu \right) } W \left|  \nabla f  \right| \right) d \mu \ge \int  \mathcal F \left( \left| f - m_f \right| \right) d \mu,
\label{wpoincare}
\end{equation}
is more natural. The next Proposition goes in this way.
\begin{prop}

Let $\mu$ a probability measure with a support $\Omega \subseteq \R^n$ and let $\omega: \Omega \to \R_+$ a function. If we assume that there exists $h \left( \mu \right) > 0$ such that
\begin{equation}
\int_{\Omega} \left| \frac{1}{h \left( \mu \right)} \nabla f \right| \, \omega \, d \mu \ge \int_{\Omega} \left| f - m_f \right| d \mu
\label{wepoincare}
\end{equation}
for every smooth enough function $f \in L^1 \left( \mu \right)$ then the following inequality
$$ \int_{\Omega} \mathcal F \left( \frac{1}{h \left( \mu \right)} \omega \left| \nabla f \right| \right) d \mu \ge \int_{\Omega} \mathcal F \left( \left| f - m_f \right| \right) d \mu $$
holds.

\end{prop}

The proof is identical to the one of Bobkov-Houdr\'e~\cite{bh} (see \cite{ce}). In the next section, we give an example whereinequality~\eqref{weightedpoincare} is fulfilled.

\subsection{Example of weighted Poincar\'{e} inequality}

Let us recall the result we will use.

\begin{thm}{\cite{ble}}

Let $\kappa \in \left( - \infty , 0 \right]$ and let $\mu$ a $\kappa$-concave measure defined on $\R^n$ (i.e. with the notation introduced in Introduction, we are in {\bf Case 2}). Let $m=\exp \left( \int_{\R^n} \log \left( \left| x \right| \right) d \mu \left( x \right) \right)$ (note that $m$ is finite). Then, for any non-empty Borel sets $A$ and $B$ in $\R^n$ located at distance $\epsilon=\mathrm{dist} \left( A , B \right) > 0$

\begin{equation}
\mu \left( A \right) \mu \left( B \right) \le \frac{C_{\kappa}}{2 \epsilon} \int_{\R^n \backslash \left( A \cup B \right)} \left( m - \kappa \left| x \right| \right) d \mu \left( x  \right)
\label{boble}
\end{equation}
with $C_{\kappa}$ depending continuously in $\kappa$ in the indicated range.

\end{thm}
Thanks to~\eqref{boble}, let us deduce a weighted Poincar\'e inequality. Let us take $K$ a non-empty compact set with smooth enough boundary and if set $A=K \backslash S$ and $B=\left( \R^n \backslash K \right) \backslash S$ where $S$ is the closure of the $\frac{\epsilon}{2}$-neighborhood of $\partial K$ in~\eqref{boble} and letting $\epsilon \to 0,$ we have:

\begin{equation}
\mu \left( K \right) \left( 1 - \mu \left( K \right) \right) \le \frac{C_{\kappa}}{2} \mu_{\omega}^+ \left( K \right),
\label{cheeger1}
\end{equation}
where $\omega \left( x \right)=m - \kappa \left| x \right|$ and $m=\exp \left( \int_{\R^n} \log \left( \left| x \right| \right) d \mu \left( x \right) \right).$ If we use the coarea formula in~\eqref{cheeger1}, we finally have

\begin{equation}
\int_{\R^n} \left| f \left( x \right) - m_f \right| d \mu \left( x \right) \le \frac{C_{\kappa}}{2} \int_{\R^n} \left| \nabla f \left( x \right) \right| \omega \left( x \right) d \mu \left( x \right).
\label{cheeger2}
\end{equation}
Now, we give an example of density $\rho_{\beta}$ such that the measure $\rho_{\beta} \left( x \right) d x$ $\kappa \left( n \right)$-concave which satisfy~\eqref{weightedpoincare}. For this, let for $\beta > n,$ $\rho_{\beta } \left( x \right) = \frac{1}{Z_{\beta}} \left( 1 + \left| x \right|^2 \right)^{-\beta}= \frac{1}{Z_{\beta}} W \left( x \right)^{- \beta}$ where $Z_{\beta}= \int_{\R^n } \left( 1 + \left| x \right|^2 \right)^{- \beta} d x.$ Recalling the notation we used in the introduction, we have $\beta= - \frac{1}{\kappa}=\frac{n \kappa \left( n \right) - 1}{\kappa \left( n \right)}$ then $\kappa \left( n \right) = \frac{1}{ n - \beta}.$ Then, the measure $\rho_{\beta} \left( x \right) d x$ is $\kappa$-concave with $\kappa=\frac{1}{ n - \beta}$. If $m=\exp \left( \int \log \left( \left| x \right| \right) \rho_{\beta} \right),$ using~\eqref{cheeger2}, we can write, for all function $f : \R^n \to \R$ smooth enough (noting $C_{\beta}$ instead of $C_{\kappa}$)

\begin{eqnarray*}
\int \left| f \left( x \right) - m_f \right| \rho_{\beta} &\le& \frac{C_{\beta}}{2} \int \left| \nabla f \left( x \right) \right| \left( m + \frac{1}{\beta - n } \left| x \right| \right) \rho_{\beta} \\
&\le&  \frac{C_{\beta}}{2} \max \left\{  m , \frac{1}{\beta - n } \right\} \int \left| \nabla f \left( x \right) \right| \left( 1 +  \left| x \right| \right) \rho_{\beta}.
\end{eqnarray*}
Proposition 1 gives

$$\int \mathcal F \left( \left| f \left( x \right) - m_f \right| \right) \rho_{\beta} \le \int \mathcal F \left(  \frac{C_{\beta}}{2} \max \left\{  m , \frac{1}{\beta - n }   \right\}  \left| \nabla f \right| \left( 1 + \left| x \right| \right)  \right)  \rho_{\beta }.$$
Remarking that $\mathcal F \left( ab \right) \le \max \left\{ a, a^2 \right\} \mathcal F \left( b \right)$ and $\left( 1 + \left| x \right| \right)^2 \le 3 \left( 1 + \left| x \right|^2 \right),$ this gives

$$\frac13 \int \mathcal F \left( \left| f \left( x \right) - m_f \right| \right) \rho_{\beta } \le  \int \mathcal F \left(  \frac{C_{\beta}}{2} \max \left\{  m , \frac{1}{\beta - n }  \right\}  \left| \nabla f \right|  \right) \left( 1 + \left| x \right|^2 \right) \rho_{\beta }.$$
Since $for t \ge 0,$ $\mathcal F \left( t / 12 \right) \le \frac13 \mathcal F \left( t \right),$ we have

$$ \int \mathcal F \left( \frac{1}{12} \left| f \left( x \right) - m_f \right| \right) \rho_{\beta,W} \le  \int \mathcal F \left(  \frac{C_{\beta}}{2} \max \left\{  m , \frac{1}{\beta - n }   \right\}  \left| \nabla f \right|  \right) W \rho_{\beta },$$
or equivalently

\begin{equation}
\int \mathcal F \left(  \frac{1}{6 C_{\beta} \max \left\{  m ,  \frac{1}{\beta - n }  \right\} } \left| f - m_f \right| \right) \rho_{\beta} \le \int \mathcal F \left( \left| \nabla f \right| \right) W \rho_{\beta}.
\label{weightedpoincare2}
\end{equation}
Thus~\eqref{weightedpoincare2} provides an example where~\eqref{weightedpoincare} is satisfied with the constant $h_W \left( \rho_{\beta , W} \right) = \frac{1}{6 C_{\beta} \max \left\{  m ,  \frac{1}{\beta - n }  \right\} } > 0.$ To conclude properly, one can give an estimation of $m.$ If we note, for $q \ge 0$ 

$$m_q=\left( \int_{\R^n} \left| x \right|^q \rho_{\beta} \left( x \right) d x \right)^{1 / q} ,$$
then we have the following lines

\begin{eqnarray*}
m &\le& m_1 \\
&=& \frac{ \int_{\R^n}  \frac{ \left| x \right| }{ \left( 1 + \left| x \right|^2 \right)^{\beta} } d x  }{  \int_{\R^n}  \frac{ 1 }{ \left( 1 + \left| x \right|^2 \right)^{\beta} } d x } \\
&=& \frac{ \int_0^{+\infty}  \frac{r^n}{ \left( 1 + r^2 \right)^{\beta} }  d r }{  \int_0^{+\infty}  \frac{r^{n - 1} }{ \left( 1 + r^2 \right)^{\beta} }  d r }.
\end{eqnarray*}
It remains to give an estimation of 

$$\frac{ \int_0^{+\infty}  \frac{r^n}{ \left( 1 + r^2 \right)^{\beta} }  d r }{  \int_0^{+\infty}  \frac{r^{n - 1} }{ \left( 1 + r^2 \right)^{\beta} }  d r }.$$
If we note

$$I_n \left( \beta \right) = \int_0^{+ \infty} \frac{r^n}{ \left( 1 + r^2 \right)^{\beta} }  d r=\int_0^{+\infty} r^n e^{- \beta \log \left( 1 + r^2 \right)} dr,$$
one can have, thanks to Laplace's method (see~\cite{di})

$$I_n \left( \beta \right)  \underset{\beta \to + \infty}{\sim} \frac12 \Gamma \left( \frac{n + 1}{2} \right) \beta^{- \frac{n+ 1}{2} }$$
and

$$\frac{ \int_0^{+\infty}  \frac{r^n}{ \left( 1 + r^2 \right)^{\beta} }  d r }{  \int_0^{+\infty}  \frac{r^{n - 1} }{ \left( 1 + r^2 \right)^{\beta} }  d r }   \underset{\beta \to + \infty}{\sim} \frac{\Gamma \left( \frac{n + 1}{2}  \right)}{  \Gamma \left( \frac{n }{2}  \right) } \sqrt{\beta}.$$
Since $\frac{\Gamma \left( \frac{n + 1}{2} \right)}{\Gamma \left( \frac{n}{2} \right)} \underset{n \to +\infty}{\sim} \sqrt{\frac{n}{2}},$ we have $m \le C \sqrt{n \beta}$ for some numerical constant $C$ and for all $\beta \ge n,$ this finally gives

$$h_W \left( \rho_{\beta , W} \right) \ge \frac{1}{6 C_{\beta}  \max \left\{  C  \sqrt{n \beta} , \frac{1}{\beta - n}  \right\} }.$$

\vskip1cm 
\noindent Erik Thomas \\
  Institut de Math\'ematiques de Jussieu, \\ Universit\'e Pierre et Marie Curie - Paris 6, \\
  75252 Paris Cedex 05, France. \\
 \verb?erik.thomas@imj-prg.fr?

\end{document}